\documentclass[11pt]{article}
\usepackage[top=1in,bottom=1.1in,left=1in,right=1in]{geometry}
\usepackage{amssymb}
\usepackage{amsmath}
\usepackage{amsthm}
\usepackage{verbatim}
\usepackage{calrsfs}
\DeclareMathAlphabet{\oldcal}{OMS}{zplm}{m}{n}
\usepackage{hyperref}

\newtheorem{theorem}{Theorem}[section]

\newtheorem{remark}[theorem]{Remark}
\newtheorem{lemma}[theorem]{Lemma}

\newcommand{\ds}{\displaystyle}
\newcommand{\eps}{\varepsilon}
\DeclareMathOperator*{\argmin}{argmin}

\begin{document}

\title{Proximal minimization in CAT$(\kappa)$ spaces\footnote{After the online publication of this paper, we realized that the proximal point algorithm and its splitting version discussed here had been previously obtained in \cite{OhtPal15} in the setting of CAT$(\kappa)$ spaces using in the definition of the resolvent the squared distance function, see \eqref{def-res-0}.}}
\author{Rafa Esp\'{i}nola${}^1$, Adriana Nicolae${}^{1,2}$ \\[0.2cm]
\footnotesize ${}^1$ Department of Mathematical Analysis, University of Seville\\
\footnotesize Apdo. 1160, 41080 Sevilla, Spain\\[0.1cm]
\footnotesize ${}^2$ Department of Mathematics, Babe\c{s}-Bolyai University, \\
\footnotesize Kog\u{a}lniceanu 1, 400084 Cluj-Napoca, Romania \\[0.1cm]  
\footnotesize E-mails: espinola@us.es, anicolae@math.ubbcluj.ro
}

\date{}
\maketitle

\begin{abstract}
In this note, we provide convergence results for the proximal point algorithm and a splitting variant thereof in the setting of CAT$(\kappa)$ spaces with $\kappa > 0$ using a recent definition for the resolvent of a convex, lower semi-continuous function due to Kimura and Kohsaka (J. Fixed Point Theory Appl. 18 (2016), 93--115).
\end{abstract}



\section{Introduction}

In Hilbert spaces, the proximal point algorithm originates from Martinet \cite{Mar70} and Rockafellar \cite{Roc76} and is a well-known method used for minimizing convex, lower-semicontinuous functions. More recently, this algorithm and generalizations thereof have been introduced in nonlinear settings too such as Riemannian manifolds of nonpositive sectional curvature \cite{FerOli02,LiLopMar09}, of sectional curvature bounded above by $\kappa > 0$ \cite{LiYao12} or even in geodesic spaces of nonpositive curvature in the sense of Alexandrov (also known as CAT$(0)$ spaces) \cite{Bac13}.

Let $X$ be a complete CAT$(0)$ space and $f : X \to (-\infty,\infty]$ be a proper, convex and lower semi-continuous function. The proximal point algorithm generates a sequence $(x_n)$ starting from a point $x_0 \in X$  by the following rule: $x_{n+1} = J^f_{\lambda_n}(x_n)$, where $(\lambda_n)$ is a sequence of positive real numbers and for $\lambda > 0$,  
\begin{equation}\label{def-res-0}
J^f_{\lambda}(x) = \argmin_{y \in X}\left[f(y) + \frac{1}{\lambda}d(y,x)^2\right], \qquad x \in X.
\end{equation}
The mapping $J^f_\lambda : X \to X$, called the resolvent of $f$, is well-defined in this context and was studied by Jost \cite{Jos95,Jos97} and Mayer \cite{May98} in connection to the theory of generalized harmonic maps. If there is no ambiguity concerning the function $f$, we usually just write $J_\lambda$. One of the remarkable properties of the resolvent is the fact that it is firmly nonexpansive in the sense of \cite{AriLeuLop14}. Moreover, the set of fixed points of $J_\lambda$ is precisely the set of minimum points of $f$. Considering a suitable notion of weak convergence that goes back to Lim \cite{Lim76} and is also referred to as $\Delta$-convergence, Ariza-Ruiz, Leu\c stean and L\'{o}pez-Acedo \cite{AriLeuLop14} showed that for any $\lambda > 0$ and $x \in X$, the sequence of Picard iterates $(J_\lambda^n (x))$ $\Delta$-converges to a minimum point of $f$ (provided such a point exists). Ba\v{c}\'{a}k \cite{Bac13} proved that if $f$ attains its minimum, then the sequence $(x_n)$ generated by the proximal point algorithm $\Delta$-converges to a minimum point of $f$. Motivated by results of Bertsekas \cite{Ber11} in the Euclidean setting, Ba\v{c}\'{a}k \cite{Bac14a} additionally studied in the context of CAT$(0)$ spaces a splitting proximal point algorithm  for finding a minimum point of a function that can be written as a finite sum of convex, lower semi-continuous functions and applied his findings to the computation of the geometric median and the Fr\'{e}chet mean of a finite set of points.

Very recently, Kimura and Kohsaka \cite{KimKoh16} introduced in  CAT$(\kappa)$ spaces with $\kappa > 0$ the resolvent of a convex, lower-semicontinuous function $f$ as an instance of so-called firmly spherically nonspreading mappings. They showed that, under appropriate boundedness conditions, the Picard iterates of the resolvent of $f$ $\Delta$-converge to a minimum point of $f$. In this paper we use this definition to study in the setting of CAT$(\kappa)$ spaces with $\kappa > 0$ the convergence of the corresponding versions of the proximal point and the splitting proximal point algorithms discussed in \cite{Bac13,Bac14a} in CAT$(0)$ spaces. 

Finally, we would like to point out that after the submission of this paper, it was brought to our attention that Kimura and Kohsaka have also independently submitted a recent joint work (which was accepted in the meantime, see \cite{KimKoh16a}) on the proximal point algorithm in CAT$(\kappa)$ spaces.

\section{Preliminaries}
Let $(X,d)$ be a metric space. A {\it geodesic path} joining $x,y \in X$ is a mapping $c:[0,l] \subseteq \mathbb{R} \to X$ such that $c(0) = x, c(l) = y$ and $d\left(c(t),c(t^{\prime})\right) = \left|t - t^{\prime}\right|$ for every $t,t^{\prime} \in [0,l]$. The image $c\left([0,l]\right)$ of $c$ is called a {\it geodesic segment} from $x$ to $y$. A point $z\in X$ belongs to such a geodesic segment if there exists $t\in [0,1]$ such that $d(z,x)= td(x,y)$ and $d(z,y)=(1-t)d(x,y)$, and in this case we write $z=(1-t)x+ty$. We say that $(X,d)$ is a {\it geodesic space} if every two points in $X$ can be joined by a geodesic path. A subset $C$ of $X$ is {\it convex} if given two points of $C$, any geodesic segment joining them is contained in $C$. A rigorous introduction to geodesic spaces is provided in \cite{Bri99}.

Let $(X,d)$ be a geodesic space. Having $C \subseteq X$ convex and $f : C \to (-\infty,\infty]$, the {\it domain} of $f$ is defined by $\text{dom} f = \{x \in C \mid f(x) < \infty\}$. The function $f$ is called {\it proper} if $\text{dom} f \ne \emptyset$. We say that $f$ is {\it convex} if for every $x,y \in C$ and $t \in [0,1]$, $f((1-t)x+ty) \le (1-t)f(x) + tf(y)$. The function $f$ is {\it uniformly convex} on $\text{dom} f$ if there exists a nondecreasing function $\delta : [0,\infty) \to [0,\infty]$ vanishing only at $0$ such that for every $x,y \in \text{dom} f$ and $t \in [0,1]$,
\[f((1-t)x+ty) \le (1-t)f(x) + tf(y) - t(1-t)\delta(d(x,y)).\]
One can prove that this is in fact equivalent to the following condition (see also \cite{Zal83} where uniformly convex functions are studied in Banach spaces): for every $\eps > 0$ there exists $\delta > 0$ such that if $x,y \in \text{dom} f$ with $d(x,y) \ge \eps$, then
\[f\left(\frac{1}{2}x + \frac{1}{2}y\right) \le \frac{1}{2}f(x) + \frac{1}{2}f(y) - \delta.\]

Fix $\kappa \in \mathbb{R}$ and denote by $M^2_{\kappa}$ the complete, simply connected model surface of constant sectional curvature $\kappa$. A {\it comparison triangle} for a geodesic triangle $\Delta(x_1,x_2,x_3)$ in $X$ is a triangle $\overline{\Delta}=\Delta(\overline{x}_1, \overline{x}_2, \overline{x}_3)$ in $M^2_{\kappa}$ such that $d(x_i,x_j) = d_{M^2_{\kappa}}(\overline{x}_i,\overline{x}_j)$ for $i,j \in \{1,2,3\}$. A geodesic triangle $\Delta$ is said to satisfy the {\it CAT$(\kappa)$ inequality} if for every comparison triangle $\overline{\Delta}$ of $\Delta$ and every $x,y \in \Delta$ we have that $d(x,y) \le d_{M^2_{\kappa}}(\overline{x},\overline{y})$, where  $\overline{x},\overline{y} \in \overline{\Delta}$ are the comparison points of $x$ and $y$, i.e., if $x = (1-t)x_i + tx_j$ then $\overline{x} = (1-t)\overline{x}_i + t\overline{x}_j$. A metric space is called a {\it CAT$(\kappa)$ space} if every two points (at distance less than $\pi/\sqrt{\kappa}$ for $\kappa > 0$) can be joined by a geodesic path and every geodesic triangle (having perimeter less than $2\pi/\sqrt{\kappa}$ for $\kappa > 0$) satisfies the CAT$(\kappa)$ inequality.

Suppose in the sequel that $(X,d)$ is a complete CAT$(\kappa)$ space with $\kappa > 0$ such that for every $v,w \in X$, $d(v,w) < \pi/(2\sqrt{\kappa})$. For $x,y,z \in X$ with $y \ne z$ and $t \in [0,1]$, the following inequality 
\begin{align} \label{ineq-CAT-k}
\begin{split}
\cos\left(\sqrt{\kappa}d((1-t)y+tz, x)\right) & \ge \frac{\sin\left(\sqrt{\kappa}(1-t)d(y,z)\right)}{\sin\left(\sqrt{\kappa}d(y,z)\right)}\cos \left(\sqrt{\kappa}d(y, x)\right)\\
& \quad + \frac{\sin\left(\sqrt{\kappa}td(y,z)\right)}{\sin\left(\sqrt{\kappa}d(y,z)\right)}\cos \left(\sqrt{\kappa}d(z, x)\right)
\end{split}
\end{align}
is an immediate consequence of the spherical law of cosines.

Let $(x_n)$ be a sequence in $X$. For $x \in X$, set $\ds r(x,(x_n)) = \limsup_{n \to \infty}d(x,x_n)$. The {\it asymptotic radius} of $(x_n)$ is given by $\ds r((x_n)) = \inf\{r(x,x_n)\mid x \in X\}$ and the {\it asymptotic center} of $(x_n)$ is  the set $\ds A((x_n))=\left\{x\in X:\limsup_{n\to\infty} d(x,x_n)=r((x_n))\right\}$. The sequence $(x_n)$ is said to {\it $\Delta$-converge} to $x\in X$ if $x$ is the unique point in the asymptotic center of every subsequence of $(x_n)$. In this case $x$ is called the {\it $\Delta$-limit} of $(x_n)$. Assume next that $r((x_n)) < \pi/(2\sqrt{\kappa})$. Then $A((x_n))$ is a singleton (see \cite[Proposition 4.1]{EspFer09}) and $(x_n)$ has a $\Delta$-convergent subsequence (see \cite[Corollary 4.4]{EspFer09}). Moreover, if $(x_n)$ $\Delta$-converges to some $x \in X$ and  $f : X \to (-\infty,\infty]$ is a proper, convex lower semi-continuous function, then $\ds f(x) \le \liminf_{n \to \infty} f(x_n)$ (see \cite[Lemma 3.1]{KimKoh16}).

Any proper, convex and lower semi-continuous function $f : X \to (-\infty,\infty]$ is bounded below (see \cite[Theorem 3.6]{KimKoh16}). If there exists a sequence $(x_n)$ such that $\ds \lim_{n \to \infty}f(x_n) = \inf_{y \in X}f(y)$ and $r((x_n)) < \pi/(2\sqrt{\kappa})$, then $f$ attains its minimum, i.e., there exists $z \in X$ such that $\ds f(z) = \inf_{x \in X} f(x)$ and we call $z$ a minimum point of $f$. Indeed, denote $\ds r = \inf_{y \in X}f(y)$ and consider the sets $C_p = \left\{y \in X \mid f(y) \le r + 1/p \right\}$. For any $p \ge 1$, the sequence $(x_n)$ will eventually be contained in $C_p$. In addition, one can easily see that $(C_p)$ is a decreasing sequence of nonempty, closed and convex sets, so, by \cite[Corollary 3.6]{EspFer09}, $\ds\bigcap_{p \ge 1}C_p \ne \emptyset$. Any point in this intersection is a minimum point of $f$. In particular, if $\text{diam}(X) < \pi/(2\sqrt{\kappa})$, then $f$ always has a minimum point. Note also that, if nonempty, the set of minimum points of $f$ has a unique closest point to any given point in $X$ because it is closed and convex. A detailed discussion on convex analysis in CAT$(0)$ spaces can be found in \cite{Bac14}.

In \cite{KimKoh16}, Kimura and Kohsaka define and study properties of the resolvent for a proper, convex and lower semi-continuous function $f : X \to (-\infty,\infty]$. Consider first, for a fixed $x \in X$, the following convex functions (see \cite[Lemma 4.1]{KimKoh16})
\begin{equation} \label{def-Psi-1}
\Psi^1_x: X \to [1/\kappa, \infty), \quad \Psi^1_x(y) = \frac{1}{\kappa\cos \left(\sqrt{\kappa}d(y,x)\right)},
\end{equation}
\begin{equation} \label{def-Psi-2}
\Psi^2_x: X \to [-1/\kappa, 0), \quad  \Psi^2_x(y) = -\frac{\cos \left(\sqrt{\kappa}d(y,x)\right)}{\kappa},
\end{equation}
and
\begin{equation} \label{def-Psi}
\Psi_x: X \to [0, \infty), \quad \Psi_x(y) = \Psi^1_x(y) + \Psi^2_x(y).
\end{equation}
Then, for $\lambda > 0$, the {\it resolvent} of $f$ is defined by
\begin{equation} \label{def-res-pos}
J_{\lambda}^f(x) = \argmin_{y \in X}\left[f(y) + \frac{1}{\lambda}\Psi_x(y)\right], \qquad x \in X.
\end{equation}
This mapping is well-defined (see \cite[Theorem 4.2]{KimKoh16}) and, when $k \searrow 0$, one actually recovers the definition \eqref{def-res-0} of the resolvent in CAT$(0)$ spaces. Moreover, if $C \subseteq X$ is nonempty, closed and convex, then the indicator function $\delta_C : X \to [0,\infty]$,
\[
\delta_C(x):=\left\{
\begin{array}{ll}
0, & \mbox{if } x \in C,\\
\infty, & \mbox{otherwise},
\end{array}
\right.
\]
is proper, convex and lower semi-continuous and for any $\lambda > 0$, $J_\lambda^{\delta_C}$ is the metric projection onto $C$, as is the case for the resolvent of $\delta_C$ in any CAT$(0)$ space (see also \cite[Remark 4.4]{KimKoh16}).

In \cite[Theorem 4.6]{KimKoh16} it is shown that the resolvent $J_\lambda$ is firmly spherically nonspreading, that is, for any $x,z \in X$, the following inequality holds
\begin{align*}
& \left(\cos\left(\sqrt{\kappa}d(x,J_\lambda x)\right) + \cos\left(\sqrt{\kappa}d(z, J_\lambda z)\right)\right)\cos^2\left(\sqrt{\kappa}d(J_\lambda x, J_\lambda z)\right)\\
& \quad \ge 2\cos\left(\sqrt{\kappa}d(J_\lambda x, z)\right)\cos\left(\sqrt{\kappa}d(x, J_\lambda z)\right).
\end{align*}
Furthermore, if $\text{Fix}(J_\lambda) \ne \emptyset$, then, by \cite[Theorem 4.6.(i)]{KimKoh16}, it follows that for every $x \in X$ and $z \in \text{Fix}(J_\lambda)$,
\[\cos \left(\sqrt{\kappa}d(J_\lambda x,z)\right)\cos \left(\sqrt{\kappa}d(x,J_\lambda x)\right) \ge \cos \left(\sqrt{\kappa}d(x,z)\right),\]
a condition which is satisfied by the projection mapping onto closed and convex subsets (see \cite[Proposition 3.5]{EspFer09}). 

The following result will be used in the next section.

\begin{lemma}\label{lemma-splitting}
Let $(a_j)$ and $(b_j)$ be sequences of nonnegative real numbers such that $(a_j)$ is bounded, $\sum_{j = 0}^\infty b_j < \infty$ and there exists $j_0 \in \mathbb{N}$ such that for all $j \ge j_0$, $a_{j+1} \ge a_j  - b_j$. Then $(a_j)$ is convergent.
\end{lemma}
\begin{proof}
For all $m,n \in \mathbb{N}$ with $n \ge m \ge j_0$, $\ds a_{n+1} \ge a_m - \sum_{j=m}^nb_j$. Thus, $\ds \liminf_{n \to \infty}a_n \ge a_m -  \sum_{j=m}^\infty b_j$, from where $\ds \liminf_{n \to \infty}a_n \ge \limsup_{m \to \infty} a_m$, which shows that $(a_j)$ is convergent.
\end{proof}

\section{Main results}

Let $(X,d)$ be a complete CAT$(\kappa)$ space with $\kappa > 0$ such that for every $v,w \in X$, $d(v,w) < \pi/(2\sqrt{\kappa})$ and suppose $f : X \to (-\infty,\infty]$ is a proper, convex and lower semi-continuous function. The following inequality also appears in the proof of \cite[Theorem 4.6]{KimKoh16} in a more particular form.

\begin{lemma} \label{lemma-ineq-pos}
If $\lambda > 0$ and $J_\lambda$ is defined by \eqref{def-res-pos}, then for  $x, z \in X$ we have that
\begin{align*}
\lambda\left(f(J_\lambda x) - f(z)\right) & \le 2\left(1+\frac{1}{\cos^2 \left(\sqrt{\kappa}d(x,J_\lambda x)\right)}\right) \\
& \quad \times 
 \left(\cos \left(\sqrt{\kappa}d(z,J_\lambda x)\right)\cos \left(\sqrt{\kappa}d(x,J_\lambda x)\right) - \cos \left(\sqrt{\kappa}d(z,x)\right)\right).
\end{align*}
\end{lemma}
\begin{proof}
If $z =  J_\lambda x$, the inequality holds with equality. Otherwise, let $a = \sqrt{\kappa}d(x,z)$, $b = \sqrt{\kappa}d(x,J_\lambda x)$, $c = \sqrt{\kappa}d(z,J_\lambda x)$ and $e = \sqrt{\kappa}d((1-t)z+tJ_\lambda x, x)$, where $t \in (0,1)$. Since
\begin{align*}
f(J_\lambda x) + \frac{1}{\lambda}\left(\frac{1}{\cos b} - \cos b\right) & \le f((1-t)z+tJ_\lambda x) + \frac{1}{\lambda}\left(\frac{1}{\cos e} - \cos e\right)\\
& \le (1-t) f(z) + t f(J_\lambda x) + \frac{1}{\lambda}\left(\frac{1}{\cos e} - \cos e\right)
\end{align*}
and, by \eqref{ineq-CAT-k},
\[\cos e \ge \frac{\sin ((1-t)c)}{\sin c}\cos a + \frac{\sin (tc)}{\sin c}\cos b\]
we obtain that
\begin{align*}
\lambda(1-t)\left(f(J_\lambda x) - f(z)\right) &\le \frac{\cos b\left(\sin c - \sin(tc)\right) - \sin ((1-t)c)\cos a}{\cos b\left(\sin((1-t)c)\cos a + \sin(tc)\cos b\right)} \\
& \quad + \left(1 - \frac{\sin (tc)}{\sin c} \right) \cos b - \frac{\sin ((1-t)c)}{\sin c}\cos a.
\end{align*}
Dividing by $(1-t)$ and letting $t \nearrow 1$, 
\[\lambda\left(f(J_\lambda x) - f(z)\right) \le \frac{c}{\sin c}\left(1 + \frac{1}{\cos^2 b}\right)\left(\cos c \cos b - \cos a\right).\]
Using the fact that for any $\alpha \in [0,\pi/2)$, $\sin \alpha \ge \alpha/2$, we obtain the desired inequality.
\end{proof}

Consider the following variant of the proximal point algorithm where one uses the resolvent defined by \eqref{def-res-pos}: given $(\lambda_n)$ a sequence of positive real numbers and $x_0 \in X$, define the sequence $(x_n)$ in $X$ by
\begin{equation} \label{alg-ppa-pos}
x_{n+1} = J_{\lambda_n}(x_{n}) = \argmin_{y \in X}\left[f(y) + \frac{1}{\lambda_n}\Psi_{x_{n}}(y)\right].
\end{equation}
The next result shows that the sequence $(x_n)$ defined above $\Delta$-converges to a minimum point of $f$ (provided such a point exists) and constitutes a counterpart of \cite[Theorem 1.4]{Bac13} from the context of CAT$(0)$ spaces.

\begin{theorem}
Let $(X,d)$ be a complete CAT$(\kappa)$ space with $\kappa > 0$ such that for every $v,w \in X$, $d(v,w) < \pi/(2\sqrt{\kappa})$. Suppose $f : X \to (-\infty,\infty]$ is a proper, convex and lower semi-continuous function which attains its minimum. Then, given any $x_0 \in X$ and any sequence of positive real numbers $(\lambda_n)$ with $\ds \sum_{n \ge 0}\lambda_n = \infty$, the sequence $(x_n)$ defined by \eqref{alg-ppa-pos} $\Delta$-converges to a minimum point of $f$.
\end{theorem}
\begin{proof}
Let $z$ be a minimum point of $f$. By Lemma \ref{lemma-ineq-pos}, we have that for every $n \in \mathbb{N}$,
\[\cos \left(\sqrt{\kappa}d(z,x_{n})\right) \le \cos \left(\sqrt{\kappa}d(z,x_{n+1})\right)\cos \left(\sqrt{\kappa}d(x_{n},x_{n+1})\right) \le \cos \left(\sqrt{\kappa}d(z,x_{n+1})\right).\]
This yields $d(z,x_{n+1}) \le d(z,x_{n})$, so $(x_n)$ is Fej\'{e}r monotone with respect to the set of minimum points of $f$. Moreover, $\ds \lim_{n \to \infty}d(z,x_n) \le d(z,x_0) < \pi/(2\sqrt{\kappa})$ and $\ds \lim_{n \to \infty}d(x_{n},x_{n+1}) = 0$, hence there exists $n_0 \in \mathbb{N}$ such that for each $n \ge n_0$, $1/\cos^2\left(\sqrt{\kappa}d(x_{n},x_{n+1})\right) < 2$. Note also that
\[f(x_{n+1}) + \frac{1}{\lambda_n}\Psi_{x_{n}}(x_{n+1}) \le f(x_{n}),\]
which shows that $\left(f(x_n)\right)$ is nonincreasing. In addition, again by Lemma \ref{lemma-ineq-pos}, we get that for all $n \ge n_0$,
\[\lambda_n\left(f(x_{n+1}) - f(z)\right) \le 6\left(\cos \left(\sqrt{\kappa}d(z,x_{n+1})\right) - \cos \left(\sqrt{\kappa}d(z,x_{n})\right)\right).\]
Thus, for $m \ge n_0$,
\begin{align*}
\left(f(x_{m+1}) - f(z)\right) \sum_{n=n_0}^m \lambda_n & \le \sum_{n=n_0}^m \lambda_n\left(f(x_{n+1}) - f(z)\right) \\
& \le 6\left(\cos \left(\sqrt{\kappa}d(z,x_{m+1})\right) - \cos \left(\sqrt{\kappa}d(z,x_{n_0})\right)\right) \le 6,
\end{align*}
from where 
\[f(x_{m+1}) \le f(z) + \frac{6}{ \sum_{n=n_0}^m \lambda_n}\]
and so $\ds \lim_{m \to \infty}f(x_m) = f(z)$.\\
Let $(x_{n_i})$ be a subsequence of $(x_n)$ which $\Delta$-converges to some $x \in X$. Then $f(x) \le \ds \liminf_{i \to \infty}f(x_{n_i}) = f(z)$, so $x$ is a minimum point of $f$. Since $(x_n)$ is Fej\'{e}r monotone with respect to the set of minimum points of $f$ and the $\Delta$-limit of every $\Delta$-convergent subsequence of $(x_n)$ is a minimum point of $f$, one can easily see that $(x_n)$ $\Delta$-converges to $x$ (see, for instance, \cite[Proposition 3.2.6]{Bac14}). 
\end{proof}

\begin{remark}
If we assume in the previous result that $\text{diam}(X) < \pi/(2\sqrt{\kappa})$, then $f$ always attains its minimum. Furthermore, if $X$ is compact, then $(x_n)$ converges to a minimum point of $f$.
\end{remark}
 
A related method for approximating a minimum point of a convex lower semi-continuous function $f$ was given in geodesic spaces by Jost \cite[Chapter 3]{Jos97} by considering a regularization of $f$ with a nonnegative, lower semi-continuous function satisfying a quantitative strict convexity condition, which is fulfilled by any uniformly convex function. We show next that the function $\Psi_x$ defined by \eqref{def-Psi} is indeed uniformly convex.

\begin{remark} \label{rmk-uc-fct}
The function $\Psi_x^1$ defined by \eqref{def-Psi-1} is uniformly convex.
\end{remark}
\begin{proof}
For $\eps > 0$, take $\delta = \eps^2/32$. Let $y,z \in X$ with $d(y,z) \ge \eps$, $t \in [0,1]$ and denote $a=\sqrt{\kappa} d(x,y)$, $b=\sqrt{\kappa} d(x,z) $ and $c = \sqrt{\kappa} d(y,z) \ge \sqrt{\kappa}\eps$. Then
\begin{align*}
\Psi^1_x\left(\frac{1}{2}y + \frac{1}{2}z\right) &\le \frac{\sin c}{\kappa\sin (c/2) (\cos a + \cos b)}= 2\cos (c/2)\frac{\cos a + \cos b}{\kappa(\cos a + \cos b)^2}\\
& \le 2\cos (c/2)\frac{\cos a + \cos b}{4\kappa\cos a \cos b} = \frac{1}{2}\cos (c/2) \left(\Psi^1_x(y) + \Psi^1_x(z)\right).
\end{align*}
Because $1 - \cos (c/2) = 2 \sin^2 (c/4) \ge 2(c/8)^2 = c^2/32$, we have that
\begin{align*}
\Psi^1_x\left(\frac{1}{2}y + \frac{1}{2}z\right) & \le \frac{1}{2}\left(1 - \frac{c^2}{32}\right) \left(\Psi^1_x(y) + \Psi^1_x(z)\right) \\
& = \frac{1}{2}\Psi^1_x(y) + \frac{1}{2}\Psi^1_x(z) - \frac{c^2}{64}\left(\Psi^1_x(y) + \Psi^1_x(z)\right).
\end{align*}
At the same time, $\Psi^1_x(y) + \Psi^1_x(z) \ge 2/\kappa$. Therefore, 
\[\Psi^1_x\left(\frac{1}{2}y + \frac{1}{2}z\right) \le \frac{1}{2}\Psi^1_x(y) + \frac{1}{2}\Psi^1_x(z) - \frac{c^2 }{32\kappa} \le \frac{1}{2}\Psi^1_x(y) + \frac{1}{2}\Psi^1_x(z) - \delta.\]
\end{proof}

Thus, $\Psi_x$ is uniformly convex as the sum of a convex and a uniformly convex function and we obtain the following immediate consequence of \cite[Theorem 3.1.1]{Jos97}.

\begin{theorem} \label{thm-conv-strong-pos}
Let $(X,d)$ be a complete CAT$(\kappa)$ space with $\kappa > 0$ such that for every $v,w \in X$, $d(v,w) < \pi/(2\sqrt{\kappa})$ and suppose $f : X \to (-\infty,\infty]$ is a proper, convex and lower semi-continuous function. If $x \in X$, $J_\lambda$ is defined by \eqref{def-res-pos} and 
\begin{equation} \label{thm-conv-strong-pos-cond1}
\limsup_{n \to \infty}d(x,J_{\lambda_n}x) < \pi/(2\sqrt{\kappa})
\end{equation}
for some sequence of positive real numbers $(\lambda_n)$ with $\ds \lim_{n \to \infty}\lambda_n = \infty$, then $\left(J_\lambda x\right)_{\lambda > 0}$ converges to a minimum point of $f$ as $\lambda \to \infty$.
\end{theorem}

\begin{remark}
If we assume above that $\text{diam}(X) < \pi/(2\sqrt{\kappa})$, then \eqref{thm-conv-strong-pos-cond1} is satisfied. Moreover, one can show that $\left(J_\lambda x\right)_{\lambda > 0}$ actually converges to the minimum point of $f$ which is closest to $x$.
\end{remark}

\cite[Chapter 4]{Jos97} studies energy functionals defined in an appropriate space of $L^2$-functions, which is a CAT$(0)$ space if the functions take values in a CAT$(0)$ space. Minimum points of such energy functionals are called generalized harmonic maps and their existence is proved via \cite[Theorem 3.1.1]{Jos97}. In a similar way, Theorem \ref{thm-conv-strong-pos} could prove to be useful for the study of energy functionals in an appropriate CAT$(\kappa)$ space of functions.

We focus next on the following splitting proximal point algorithm employed in the study of minimum points for a function $f : X \to (-\infty,\infty]$ which can be written as
\begin{equation} \label{def-fct-splitting}
f = \sum_{i=1}^N f_i, 
\end{equation}
where for each $i \in \{1, \ldots, N\}$, $f_i : X \to (-\infty,\infty]$ is proper, convex and lower semi-continuous. To this end we will use instead of the resolvent of $f$, the resolvents of the functions $f_i$,
\[J^i_\lambda (x) = \argmin_{y \in X}\left[f_i(y) + \frac{1}{\lambda}\Psi_{x}(y) \right]\]
and given $(\lambda_j)$ a sequence of positive real numbers and $x_0 \in X$, the sequence $(x_n)$ is defined by
\begin{equation} \label{alg-splitting}
x_{jN+1} = J^1_{\lambda_j}(x_{jN}), \quad x_{jN+2} = J^2_{\lambda_j}(x_{jN+1}), \quad \ldots, \quad x_{jN+N} = J^N_{\lambda_j}(x_{jN+N-1}).
\end{equation}
This method was recently studied in CAT$(0)$ spaces in \cite[Theorem 3.4]{Bac14a} and we adapt the proof strategy to our setting.

\begin{theorem}
Let $(X,d)$ be a compact CAT$(\kappa)$ space with $\kappa > 0$ such that for every $v,w \in X$, $d(v,w) < \pi/(2\sqrt{\kappa})$. Suppose $f : X \to (-\infty,\infty]$ is a function of the form \eqref{def-fct-splitting} which attains its minimum. For any $x_0 \in X$ and any sequence of positive real numbers $(\lambda_j)$ with $\ds \sum_{j \ge 0}\lambda_j = \infty$ and $\ds \sum_{j \ge 0}\lambda_j^2 < \infty$, let $(x_n)$ be defined by \eqref{alg-splitting}. If there exists $L > 0$ such that for every $j \in \mathbb{N}$ and $i \in \{1,\ldots, N\}$,
\begin{equation}\label{thm-splitting-cond1}
f_i(x_{jN}) - f_i(x_{jN+i}) \le L d(x_{jN}, x_{jN+i})
\end{equation}
and
\begin{equation}\label{thm-splitting-cond2}
f_i(x_{jN+i-1}) - f_i(x_{jN+i}) \le L d(x_{jN+i-1}, x_{jN+i}),
\end{equation}
then $(x_n)$ converges to a minimum point of $f$.
\end{theorem} 
\begin{proof}
Let $z$ be a minimum point of $f$. For $j \in \mathbb{N}$ and $i \in \{1, \ldots, N\}$, apply Lemma \ref{lemma-ineq-pos} to the function $f_i$, the resolvent $J_{\lambda_j}^i$ and the points $x_{jN+i-1}$ and $z$ to get that
\begin{align*}
\lambda_j\left(f_i(x_{jN+i}) - f_i(z)\right) & \le 2\left(1 + \frac{1}{\cos^2\left(\sqrt{\kappa}d(x_{jN+i-1},x_{jN+i})\right)}\right)\\
& \quad \times \left(\cos \left(\sqrt{\kappa}d(z,x_{jN+i})\right) - \cos \left(\sqrt{\kappa}d(z,x_{jN+i-1})\right) \right).
\end{align*}
Let $m \in  \{1, \ldots, N\}$. Because
\[f_m(x_{jN+m}) + \frac{1}{\lambda_j} \Psi_{x_{jN+m-1}}(x_{jN+m}) \le f_m(x_{jN+m-1})\]
and
\begin{align*}
\Psi_{x_{jN+m-1}}(x_{jN+m}) & = \frac{\sin^2\left(\sqrt{\kappa}d(x_{jN+m-1},x_{jN+m})\right)}{\cos \left(\sqrt{\kappa}d(x_{jN+m-1},x_{jN+m})\right)} \\
& \ge \sin^2\left(\sqrt{\kappa}d(x_{jN+m-1},x_{jN+m})\right) \ge \frac{\kappa d(x_{jN+m-1},x_{jN+m})^2}{4},
\end{align*}
by \eqref{thm-splitting-cond2}, we have that
\begin{align*}
\frac{\kappa d(x_{jN+m-1},x_{jN+m})^2}{4} &\le \lambda_j\left(f_m(x_{jN+m-1}) - f_m(x_{jN+m})\right)\\
& \le \lambda_j L d(x_{jN+m-1},x_{jN+m}),
\end{align*}
from where
\begin{equation}\label{thm-splitting-eq1}
d(x_{jN+m-1},x_{jN+m}) \le 4\lambda_jL/\kappa.
\end{equation}
Take $j_0 \in \mathbb{N}$ such that $\lambda_j \le 1$ for $j \ge j_0$ and denote $\alpha = 1 + 1/\cos^2(4L/\sqrt{\kappa})$. Then if $j \ge j_0$, $1 + 1/\cos^2\left(\sqrt{\kappa}d(x_{jN+i-1},x_{jN+i})\right) \le \alpha$ and so
\[\lambda_j\left(f_i(x_{jN+i}) - f_i(z)\right) \le 2\alpha\left(\cos \left(\sqrt{\kappa}d(z,x_{jN+i})\right) - \cos \left(\sqrt{\kappa}d(z,x_{jN+i-1})\right) \right).\]
Note that 
\[\sum_{i=1}^N\left(f_i(x_{jN+i}) - f_i(z)\right) = f(x_{jN}) - f(z) + \sum_{i=1}^N\left(f_i(x_{jN+i}) - f_i(x_{jN})\right).\]
Hence, for all $j \ge j_0$,
\begin{align*}
\lambda_j\left(f(x_{jN}) - f(z)\right) &\le 2\alpha\left(\cos \left(\sqrt{\kappa}d(z,x_{jN+N})\right) - \cos \left(\sqrt{\kappa}d(z,x_{jN})\right) \right) \\
& \quad + \lambda_j\sum_{i=1}^N\left(f_i(x_{jN}) - f_i(x_{jN+i})\right).
\end{align*}
Using \eqref{thm-splitting-eq1} we obtain that
\[d(x_{jN},x_{jN+i}) \le d(x_{jN}, x_{jN+1}) + \ldots + d(x_{jN+i-1}, x_{jN+i}) \le 4i\lambda_j L/\kappa,\]
which, by \eqref{thm-splitting-cond1}, yields that for $j \ge j_0$,
\begin{align}\label{thm-splitting-eq2}
\begin{split}
\lambda_j\left(f(x_{jN}) - f(z)\right) & \le 2\alpha\left(\cos \left(\sqrt{\kappa}d(z,x_{(j+1)N})\right) - \cos \left(\sqrt{\kappa}d(z,x_{jN})\right) \right) \\
& \quad + 2N(N+1)\lambda_j^2L^2/\kappa.
\end{split}
\end{align}
Since $z$ is a minimum point of $f$, we have that for $j \ge j_0$,
\[\cos \left(\sqrt{\kappa}d(z,x_{(j+1)N})\right) \ge  \cos \left(\sqrt{\kappa}d(z,x_{jN})\right) - N(N+1)\lambda_j^2L^2/(\kappa\alpha),\]
which, by Lemma \ref{lemma-splitting}, shows that the sequence $\left(\cos \left(\sqrt{\kappa}d(z,x_{jN})\right)\right)_j$ is convergent and so the sequence $\left(d(z,x_{jN})\right)_j$ converges too. Moreover, using \eqref{thm-splitting-eq2}, we get that $\ds \sum_{j=0}^\infty \lambda_j\left(f(x_{jN}) - f(z)\right) < \infty$. This implies that there exists a subsequence $(x_{j_lN})_l$ of $(x_{jN})$ such that $\ds \lim_{l \to \infty}f(x_{j_lN}) = f(z)$. We may assume that $(x_{j_lN})_l$ converges to some $p \in X$ (otherwise take a convergent subsequence of it). Since $f$ is lower semi-continuous, $\ds f(p) \le \lim_{l \to \infty}f(x_{j_lN}) = f(z)$, so $p$ is a minimum point of $f$, which means that $\left(d(p,x_{jN})\right)_j$ is convergent and must converge to $0$ since $(x_{j_lN})_l$ converges to $p$. Now one only needs to use \eqref{thm-splitting-eq1} to obtain that $(x_{jN+m})_j$ converges to $p$ for all $m \in \{1, \ldots, N\}$ which finally yields that $(x_n)$ converges to $p$.
\end{proof}

\begin{remark}
Other choices for the function $\Psi_x$ in the definition of the resolvent \eqref{def-res-pos} are also possible, even if the image of $\Psi_x$ is not $[0,\infty)$. For instance, it is easy to see that the resolvent is well-defined and that similar convergence results hold for the sequence generated by the proximal point algorithm when considering $\Psi_x = \Psi^1_x$ or $\Psi_x = \Psi^2_x$.  
\end{remark}

\begin{remark}
Although the proximal point algorithm as given in \cite{Bac13} can be applied in any CAT$(\kappa)$ space with $\kappa \le 0$, as before one could also consider for $\kappa < 0$ another algorithm of this type by taking, for example, in \eqref{def-res-pos} $\Psi_x: X \to [0, \infty)$, $\ds\Psi_x(y) = -\frac{1}{\kappa}\left(\cosh \left(\sqrt{-\kappa}d(y,x)\right) -\frac{1}{\cosh \left(\sqrt{-\kappa}d(y,x)\right)}\right)$. Note that when $\kappa \nearrow 0$, we obtain the squared distance function as for CAT$(0)$ spaces. It turns out that the resolvent is indeed well-defined and that analogous convergence results can be proved in this case too.
 \end{remark}
 
{\bf Acknowledgements:} 
The authors have been partially supported by DGES (MTM2015-65242-C2-1-P). A. Nicolae would also like to
acknowledge the Juan de la Cierva-incorporaci\'{o}n Fellowship Program of the Spanish Ministry of Economy and Competitiveness.

\end{document}